\documentclass{eptcs-modified}
\usepackage{amsmath}
\usepackage{amssymb}
\usepackage{amsthm}
\usepackage{amscd}

\title{Noetherian Quasi-Polish Spaces\thanks{This work was supported by JSPS Core-to-Core Program, A. Advanced Research Networks. The first author was supported by JSPS KAKENHI Grant Number 15K15940. The second author was supported by the ERC inVEST (279499)
project.}}
\author{
Matthew de Brecht
\institute{Graduate School of Human and Environmental Studies\\ Kyoto University, Japan\\}
\email{matthew@i.h.kyoto-u.ac.jp}
\and
Arno Pauly
\institute{D\'epartement d'Informatique\\ Universit\'e libre de
Bruxelles, Belgium\\}
\email{Arno.M.Pauly@gmail.com}
}
\def\titlerunning{Noetherian Quasi-Polish Spaces}
\def\authorrunning{M.~de Brecht \& A. Pauly}
\begin{document} \theoremstyle{definition}
\newtheorem{theorem}{Theorem}
\newtheorem{definition}[theorem]{Definition}
\newtheorem{problem}[theorem]{Problem}
\newtheorem{assumption}[theorem]{Assumption}
\newtheorem{corollary}[theorem]{Corollary}
\newtheorem{proposition}[theorem]{Proposition}
\newtheorem{lemma}[theorem]{Lemma}
\newtheorem{observation}[theorem]{Observation}
\newtheorem{fact}[theorem]{Fact}
\newtheorem{question}[theorem]{Question}
\newtheorem{example}[theorem]{Example}
\newcommand{\dom}{\operatorname{dom}}
\newcommand{\Dom}{\operatorname{Dom}}
\newcommand{\codom}{\operatorname{CDom}}
\newcommand{\id}{\textnormal{id}}
\newcommand{\Cantor}{\{0, 1\}^\mathbb{N}}
\newcommand{\Baire}{\mathbb{N}^\mathbb{N}}
\newcommand{\uint}{{[0,1]}}
\newcommand{\lev}{\textnormal{Lev}}
\newcommand{\hide}[1]{}
\newcommand{\mto}{\rightrightarrows}
\newcommand{\pls}{\textsc{PLS}}
\newcommand{\ppad}{\textsc{PPAD}}
\newcommand{\mlr}{\textrm{MLR}}
\newcommand{\layer}{\textrm{Layer}}
\newcommand{\blayer}{\textrm{BLayer}}
\newcommand{\sierp}{Sierpi\'nski}
\newcommand{\C}{\textrm{C}}
\newcommand{\lpo}{\textrm{LPO}}
\newcommand{\const}{\mathfrak{C}}
\newcommand{\2}{\mathbf{2}}
\newcommand{\name}[1]{\textsc{#1}}
\newcommand{\sat}{\uparrow}
\newcommand\tboldsymbol[1]{%
\protect\raisebox{0pt}[0pt][0pt]{%
$\underset{\widetilde{}}{\boldsymbol{#1}}$}\mbox{\hskip 1pt}}

\newcommand{\bolds}{\tboldsymbol{\Sigma}}
\newcommand{\boldp}{\tboldsymbol{\Pi}}
\newcommand{\boldd}{\tboldsymbol{\Delta}}
\newcommand{\boldg}{\tboldsymbol{\Gamma}}

\maketitle

\begin{abstract}
In the presence of suitable power spaces, compactness of $\mathbf{X}$ can be characterized as the singleton $\{X\}$ being open in the space $\mathcal{O}(\mathbf{X})$ of open subsets of $\mathbf{X}$. Equivalently, this means that universal quantification over a compact space preserves open predicates.

Using the language of represented spaces, one can make sense of notions such as a $\Sigma^0_2$-subset of the space of $\Sigma^0_2$-subsets of a given space. This suggests higher-order analogues to compactness: We can, e.g.~, investigate the spaces $\mathbf{X}$ where $\{X\}$ is a $\Delta^0_2$-subset of the space of $\Delta^0_2$-subsets of $\mathbf{X}$. Call this notion $\nabla$-compactness. As $\Delta^0_2$ is self-dual, we find that both universal and existential quantifier over $\nabla$-compact spaces preserve $\Delta^0_2$ predicates.

Recall that a space is called Noetherian iff every subset is compact. Within the setting of Quasi-Polish spaces, we can fully characterize the $\nabla$-compact spaces: A Quasi-Polish space is Noetherian iff it is $\nabla$-compact. Note that the restriction to Quasi-Polish spaces is sufficiently general to include plenty of examples.
\end{abstract}

\section{Introduction}

\subsubsection*{Noetherian spaces}
\begin{definition}
\label{def:noetherian}
A topological space $\mathbf{X}$ is called \emph{Noetherian}, iff every strictly ascending chain of open sets is finite.
\end{definition}

Noetherian spaces were first studied in algebraic geometry. Here, the prime motivation is that the Zariski topology on the spectrum of a Noetherian commutative ring is Noetherian (which earns the Noetherian spaces their name).

The relevance of Noetherian spaces for computer science was noted by \name{Goubault-Larrecq} \cite{goubault2}, based on their relationship to well quasiorders. Via well-structured transition systems \cite{finkel}, well quasiorders are used in verification to prove decidability of termination and related properties. Unfortunately, well quasiorders lack some desirable closure properties (the standard counterexample is due to \name{Rado} \cite{rado}), which led to the introduction of better quasiorders by \name{Nash-Williams} \cite{nashwilliams}, which is a more restrictive notion avoiding the shortcomings of well quasiorders.

Noetherian spaces generalize well-quasi orders: The Alexandrov topology on a quasi-order is Noetherian iff the quasi-order is a well-quasi order. As shown by \name{Goubault-Larrecq}  \cite{goubault}, results on the preservation of well-quasi orders under various constructions (such as Higman's Lemma or Kruskal's Tree Theorem \cite{goubault3}) extend to Noetherian spaces; furthermore, Noetherian spaces exhibit some additional closure properties, e.g.~the Hoare space of a Noetherian space is Noetherian again \cite{goubault2}. The usefulness of Noetherian spaces for verification is detailed by \name{Goubault-Larrecq}  in \cite{goubault4}.

\subsubsection*{Quasi-Polish spaces}
A countably-based topological space is called quasi-Polish if its topology can be derived from a Smyth-complete quasi-metric. Quasi-Polish spaces were introduced by \name{dB.} in \cite{debrecht6} as a joint generalization of Polish spaces and $\omega$-continuous domains in order to satisfy the desire for a unified setting for descriptive set theory in those areas (expressed e.g.~by \name{Selivanov} \cite{selivanov3}).

\subsubsection*{Synthetic DST}
Synthetic descriptive set theory as proposed by the authors in \cite{paulydebrecht2-lics} reinterprets descriptive set theory in a category-theoretic context. In particular, it provides notions of lifted counterparts to topological concepts such as open sets (e.g.~$\Sigma$-classes from descriptive set theory), compactness, and so on.

\subsubsection*{Our contributions}
In the present paper, we will study Noetherian quasi-Polish spaces. As our main result, we show that in the setting of quasi-Polish spaces, being Noetherian is the $\Delta^0_2$-analogue to compactness. We present the result in two different incarnations: Theorem \ref{theo:finitecover} states the result in the language of traditional topology. Theorem \ref{theo:nablacompact} then restates the main result in the language of synthetic topology, which first requires us to define a computable version of being Noetherian (Definition \ref{def:real}). The second instance in particular has as a consequence that universal and existential quantification over Noetherian spaces preserves $\Delta^0_2$-predicates -- and this characterizes Noetherian spaces (Proposition \ref{prop:quantifierelimination}).

\subsubsection*{Structure of the article}
In Section \ref{sec:quasipolish} we recall some results on Noetherian spaces and on quasi-Polish spaces, and then prove some observations on Noetherian quasi-Polish spaces. In particular, Theorem \ref{theo:finitecover} shows that for quasi-Polish spaces, being Noetherian is equivalent to any $\Delta^0_2$-cover admitting a finite subcover. This section requires only some basic background from topology.

Section \ref{prop:quantifierelimination} introduces the additional background material we need for the remainder of the paper, in particular from computable analysis and synthetic topology.

In Section \ref{sec:computablynoetherian} we investigate how Noetherian spaces ought to be defined in synthetic topology (\name{Escard\'o} \cite{escardo}), specifically in the setting of the category of represented spaces (\name{P.} \cite{pauly-synthetic}). As an application, we show that computable well-quasiorders give rise to $\nabla$-computably Noetherian spaces.

Our main result will be presented in Section \ref{sec:nablacompact}: The Noetherian spaces can be characterized amongst the quasi-Polish spaces as those allowing quantifier elimination over $\Delta^0_2$-statements (Theorem \ref{theo:nablacompact} and Corollary \ref{corr:quantifier}). The core idea is that just as compact spaces are characterized by $\{X\}$ being an open subset of the space $\mathcal{O}(\mathbf{X})$ of open subsets, the Noetherian spaces are (amongst the quasi-Polish) characterized by $\{X\}$ being a $\Delta^0_2$-subset of the space of $\Delta^0_2$-subsets.

Looking onwards, we briefly discuss potential future extensions of characterizations of higher-order analogues to compactness and overtness in Section \ref{sec:othernotions}.

\section{Initial observations on Noetherian quasi-Polish spaces}
\label{sec:quasipolish}

\subsection{Background on Quasi-Polish spaces}

Recall that a quasi-metric on $\mathbf{X}$ is a function $d : \mathbf{X} \times \mathbf{X} \to [0,\infty)$ such that $x = y \Leftrightarrow d(x,y) = d(y,x) = 0$ and $d(x,z) \leq d(x,y) + d(y,z)$. A quasi-metric induces a topology via the basis $(B(x,2^{-k}) := \{y \in \mathbf{X} \mid d(x,y) < 2^{-k}\})_{x \in \mathbf{X},k \in \mathbb{N}}$. A topological space is called quasi-Polish, if it is countably-based and the topology can be obtained from a Smyth-complete quasi-metric (from \name{Smyth} \cite{smyth}). For details we refer to \cite{debrecht6}, and only recall some select results to be used later on here.

\begin{proposition}[\name{dB.} \cite{debrecht6}]
\label{prop:quasisubspace}
A subspace of a quasi-Polish space is a quasi-Polish space iff it is a $\Pi^0_2$-subspace.
\end{proposition}

\begin{corollary}
\label{corr:quasisingleton}
In a quasi-Polish space each singleton is $\Pi^0_2$.
\end{corollary}

\begin{proposition}[\name{dB.} \cite{debrecht6}]
A space is quasi-Polish iff it is homoeomorphic to a $\Pi^0_2$-subspace of the Scott domain $\mathcal{P}(\omega)$.
\end{proposition}

\begin{theorem}[{\name{Heckmann} \cite{heckmann}, \name{Becher} \& \name{Grigorieff} \cite[Theorem 3.14]{becher}}]
\label{theo:bairecategory}
Let $\mathbf{X}$ be quasi-Polish. If $\mathbf{X} = \bigcup_{i \in \mathbb{N}} A_i$ with each $A_i$ being $\Sigma^0_2$, then there is some $i_0$ such that $A_{i_0}$ has non-empty interior.
\end{theorem}

Recall that a closed set is called \emph{irreducible}, if it is not the union of two proper closed subsets. A topological space is called \emph{sober}, if each non-empty irreducible closed set is the closure of a singleton.

\begin{proposition}[\name{dB.} \cite{debrecht6}]
\label{prop:soberquasipolish}
A countably-based locally compact sober space is quasi-Polish. Conversely, each quasi-Polish space is sober.
\end{proposition}

\subsection{Background on Noetherian spaces}

\begin{theorem}[\name{Goubault-Larrecq}  \cite{goubault}]
\label{theo:glnoethcharac}
The following are equivalent for a topological space $\mathbf{X}$:
\begin{enumerate}
\item $\mathbf{X}$ is Noetherian, i.e.~every strictly ascending chain of open sets is finite (Definition \ref{def:noetherian}).
\item Every strictly descending chain of closed sets is finite.
\item Every open set is compact.
\item Every subset is compact.
\end{enumerate}
\end{theorem}

As being Noetherian is preserved by sobrification\footnote{Sobrification only adds points, not open sets, and being Noetherian is only about open sets.}, we do not lose much by restricting our attention to sober Noetherian spaces. These admit a useful characterization as the upper topologies for certain well-founded partial orders. In the following we use the notation $\downarrow x := \{y \in X \mid y \prec x\}$.

\begin{theorem}[\name{Goubault-Larrecq}  \cite{goubault}]
\label{theo:glsobernoethcharac}
The following are equivalent for a topological space $\mathbf{X} = (X, \mathcal{T})$:
\begin{enumerate}
\item $\mathbf{X}$ is a sober Noetherian space.
\item There is some well-founded partial order $\prec$ on $X$ such that $\mathcal{T}$ is the upper topology induced by $\prec$ and for any finite $F \subseteq \mathbf{X}$ there is a finite $G \subseteq \mathbf{X}$ such that: \[\bigcap_{x \in F} \downarrow x = \bigcup_{y \in G} \downarrow y\]
\end{enumerate}
\end{theorem}

\begin{lemma}[\name{Goubault-Larrecq}  \cite{goubault}]
\label{lem:closedfinitelygenerated}
Every closed subset of a sober Noetherian space is the closure of a finite set.
\end{lemma}

\subsection{Some new observations}

\begin{theorem}
\label{theo:characnoethquasi}
The following are equivalent for a sober Noetherian space $\mathbf{X}$:
\begin{enumerate}
\item $\mathbf{X}$ is countable.
\item $\mathbf{X}$ is countably-based.
\item $\mathbf{X}$ is quasi-Polish.
\end{enumerate}
\begin{proof}
\begin{description}
\item[$1. \Rightarrow 2.$] By Theorem \ref{theo:glsobernoethcharac}, we can consider $\mathbf{X}$ to be equipped with the upper topology for some partial order. If $\mathbf{X}$ is countable, then any upper topology is countable, too.
\item[$2. \Rightarrow 1.$] By Theorem \ref{theo:glnoethcharac}, every open subset is compact, hence a finite union of basic open sets. Thus, a countably-based Noetherian topology is countable. A sober space with a countable topology has only countably many points.
\item[$2. \Rightarrow 3.$] As a Noetherian space is compact, we know $\mathbf{X}$ to be a countably-based sober compact space. Proposition \ref{prop:soberquasipolish} then implies $\mathbf{X}$ to be quasi-Polish.
\item[$3. \Rightarrow 2.$] By definition.
\end{description}
\end{proof}
\end{theorem}

\begin{corollary}
A subspace of a quasi-Polish Noetherian space is sober iff it is a $\Pi^0_2$-subspace.
\begin{proof}
Combine Theorem \ref{theo:characnoethquasi} with Proposition \ref{prop:quasisubspace}.
\end{proof}
\end{corollary}

The following theorem already showcases the link between being Noetherian and a $\Delta^0_2$-analogue to compactness. Its proof is split into Lemmata \ref{lemma:converingcoverse},\ref{lemma:delta02coveringfinite} and Observation \ref{obs:sigmadelta}.

\begin{theorem}
\label{theo:finitecover}
The following are equivalent for a quasi-Polish space $\mathbf{X}$:
\begin{enumerate}
\item $\mathbf{X}$ is Noetherian.
\item Every $\Delta^0_2$-cover of $\mathbf{X}$ has a finite subcover.
\item Every $\Sigma^0_2$-cover of $\mathbf{X}$ has a finite subcover.
\end{enumerate}
\end{theorem}

\begin{lemma}
\label{lemma:converingcoverse}
If a topological space $\mathbf{X}$ is not Noetherian, then it admits a countably-infinite $\Delta^0_2$-partition.
\begin{proof}
If $\mathbf{X}$ is not Noetherian, then there must be an infinite strictly ascending chain $(U_i)_{i \in \mathbb{N}}$ of open sets. Then $\{U_{i+1} \setminus U_i \mid i \in \mathbb{N}\} \cup \{U_0, \left (\bigcup_{i \in \mathbb{N}} U_i\right )^C\}$ constitutes a $\Delta^0_2$-partition with countably-infinitely many non-trivial pieces.
\end{proof}
\end{lemma}

\begin{lemma}
\label{lemma:delta02coveringfinite}
Any $\Delta^0_2$-cover of a Noetherian quasi-Polish space has a finite subcover.
\begin{proof}
Since $\mathbf{X}$ is countable we can assume the covering is
countable. By the Baire category theorem for quasi-Polish spaces (Theorem \ref{theo:bairecategory}), there is a $\Delta^0_2$-set $A_0$
in the covering such that its interior, $U_0$ is non-empty.

For $n \geq 0$, if $\mathbf{X} \neq U_n$, then we repeat the same argument with
respect to $\mathbf{X} \setminus U_n$ to get a $\Delta^0_2$-set $A_{n+1}$ in the covering with non-empty interior relative to $\mathbf{X} \setminus U_n$. Define $U_{n+1}$ to be the
union of $U_n$
and the relative interior of $A_{n+1}$. Then $U_{n+1}$ is an
open subset of $\mathbf{X}$ which strictly contains $U_n$.
Since $\mathbf{X}$ is Noetherian, eventually $\mathbf{X} = U_n$, and $A_0,\ldots,A_n$ will
yield a finite subcovering of $\mathbf{X}$.
\end{proof}
\end{lemma}

\begin{observation}
\label{obs:sigmadelta}
Any $\Sigma^0_2$-cover of a quasi-Polish space can be refined into a $\Delta^0_2$-cover, and any $\Delta^0_2$-cover is a $\Sigma^0_2$-cover.
\end{observation}

\begin{corollary}
Let $\mathbf{X}$ be a Noetherian quasi-Polish space, and let $\mathbf{X}^\delta$ be the topology induced by the $\Delta^0_2$-subsets of $\mathbf{X}$. Then $\mathbf{X}^\delta$ is a compact Hausdorff space.
\begin{proof}
That $\mathbf{X}^\delta$ is compact follows from Lemma \ref{lemma:delta02coveringfinite}. To see that it is Hausdorff, we just note that in any $T_0$-space, two distinct points can be separated by a disjoint pair of an open and a closed set -- hence by $\Delta^0_2$-sets.
\end{proof}
\end{corollary}

Recall that a topological space satisfies the $T_D$-separation axiom (cf.~\cite{aull}) iff every singleton is a $\Delta^0_2$-set.

\begin{corollary}
A Noetherian quasi-Polish space is $T_D$ iff it is finite.
\begin{proof}
If $\mathbf{X}$ is a $T_D$ space, then $\mathbf{X} = \bigcup_{x \in \mathbf{X}} \{x\}$ is a $\Delta^0_2$ covering of it. By Lemma \ref{lemma:delta02coveringfinite}, it then follows that there is a finite subcovering, which can only be identical to the original covering -- hence, $\mathbf{X}$ is finite. For the converse direction, by Corollary \ref{corr:quasisingleton} every singleton in a quasi-Polish space is $\Pi^0_2$. In a finite space, it follows that they are even $\Delta^0_2$.
\end{proof}
\end{corollary}

\begin{corollary}
An infinite Noetherian quasi-Polish space contains a $\Pi^0_2$-complete singleton.
\end{corollary}

We can also obtain the following special case of \name{Goubault-Larrecq}'s Lemma \ref{lem:closedfinitelygenerated} as a corollary of Lemma \ref{lemma:delta02coveringfinite}:

\begin{corollary}
\label{corr:closedfinitelygenerated}
Every closed subset of a quasi-Polish Noetherian space is the closure of a finite set.
\begin{proof}
Given some closed subset $A \subseteq \mathbf{X}$, consider the $\Delta^0_2$-cover $\mathbf{X} = A^C \cup \bigcup_{x \in A} \textrm{cl} \{x\}$. By Lemma \ref{lemma:delta02coveringfinite} there is some finite subcover $\mathbf{X} = A^C \cup \bigcup_{x \in F} \textrm{cl} \{x\}$, but then it follows that $A = \textrm{cl} F$.
\end{proof}
\end{corollary}

Neither being sober nor being quasi-Polish is preserved by continuous images in general. However, being Noetherian is not only preserved itself, but in its presence, so are the other properties:

\begin{proposition}
Let $\mathbf{X}$ be a Noetherian sober (quasi-Polish) space and $\sigma : \mathbf{X} \to \mathbf{Y}$ a continuous surjection. Then $\mathbf{Y}$ is Noetherian sober (quasi-Polish) space, too.
\begin{proof}
 Let $C \subseteq \mathbf{Y}$ be irreducible closed. Then $\sigma^{-1}(C)$ is closed, so by Lemma \ref{lem:closedfinitelygenerated} (or Corollary \ref{corr:closedfinitelygenerated}) there is finite $F \subseteq \mathbf{X}$ such that $\textrm{cl}(F) = \sigma^{-1}(C)$. Continuity implies $\textrm{cl}(\sigma(F)) \supseteq \sigma(\textrm{cl}(F)) = C$, hence $\textrm{cl}(\sigma(F)) = C$. Since $\sigma(F)$ is finite and $C$ is irreducible, $C$ must be equal to the closure of some element of $\sigma(F)$. Therefore, $\mathbf{Y}$ is sober.

 By Theorem \ref{theo:characnoethquasi}, for Noetherian sober spaces being quasi-Polish is equivalent to being countable, which is clearly preserved by (continuous) surjections.
\end{proof}
\end{proposition}

\section{Background}
\label{sec:background}
\subsubsection*{Computable analysis}
In the remainder of this article, we wish to explore the uniform or effective aspects of the theory of Noetherian Quasi-Polish spaces. The basic framework for this is provided by \emph{computable analysis} \cite{weihrauchd}. Here the core idea is to introduce notions of continuity and in particular continuity on a wide range of spaces by translating them from those on Baire space via the so-called representations. Our notation and presentation follows closely that of \cite{pauly-synthetic}, which in turn is heavily influenced by \name{Escard\'o}'s \emph{synthetic topology} \cite{escardo}, and by work by \name{Schr\"oder} \cite{schroder5}.

\begin{definition}
A represented space is a pair $\mathbf{X} = (X, \delta_\mathbf{X})$ where $X$ is a set and $\delta_\mathbf{X} : \subseteq \Baire \to X$ is a partial surjection. A function between represented spaces is a function between the underlying sets.
\end{definition}

\begin{definition}
For $f : \subseteq \mathbf{X} \to \mathbf{Y}$ and $F : \subseteq \Baire \to \Baire$, we call $F$ a realizer of $f$ (notation $F \vdash f$), iff $\delta_Y(F(p)) = f(\delta_X(p))$ for all $p \in \dom(f\delta_X)$, i.e.~if the following diagram commutes:
  $$\begin{CD}
\Baire @>F>> \Baire\\
@VV\delta_\mathbf{X}V @VV\delta_\mathbf{Y}V\\
\mathbf{X} @>f>> \mathbf{Y}
\end{CD}$$
A map between represented spaces is called computable (continuous), iff it has a computable (continuous) realizer.
\end{definition}

Two represented spaces of particular importance are the integers $\mathbb{N}$ and Sierpi\'nski space $\mathbb{S}$. The represented space $\mathbb{N}$ has as underlying set $\mathbb{N}$ and the representation $\delta_\mathbb{N} : \Baire \to \mathbb{N}$ defined by $\delta_\mathbb{N}(p) = p(0)$. The Sierpi\'nski space $\mathbb{S}$ has the underlying set $\{\top,\bot\}$ and the representation $\delta_\mathbb{S}$ with $\delta_\mathbb{S}(0^\omega) = \top$ and $\delta_\mathbb{S}(p) = \bot$ for $p \neq 0^\omega$.

Represented spaces have binary products, defined in the obvious way: The underlying set of $\mathbf{X} \times \mathbf{Y}$ is $X \times Y$, with the representation $\delta_{\mathbf{X} \times \mathbf{Y}}(\langle p, q\rangle) = (\delta_\mathbf{X}(p),\delta_\mathbf{Y}(q))$. Here $\langle \ , \ \rangle : \Baire \times \Baire \to \Baire$ is the pairing function defined via $\langle p, q\rangle(2n) = p(n)$ and $\langle p, q\rangle(2n+1) = q(n)$.

A central reason for why the category of represented space is such a convenient setting lies in the fact that it is cartesian closed: We have available a function space construction $\mathcal{C}(\cdot, \cdot)$, where the represented space $\mathcal{C}(\mathbf{X},\mathbf{Y})$ has as underlying set the continuous functions from $\mathbf{X}$ to $\mathbf{Y}$, represented in such a way that the evaluation map $(f, x) : \mathcal{C}(\mathbf{X},\mathbf{Y}) \times \mathbf{X} \to \mathbf{Y}$ becomes computable. This can be achieved, e.g., by letting $nq$ represent $f$, if the $n$-th Turing machine equipped with oracle $q$ computes a realizer of $f$. This also makes currying, uncurrying and composition all computable maps.

Having available to us the space $\mathbb{S}$ and the function space construction, we can introduce the spaces $\mathcal{O}(\mathbf{X})$ and $\mathcal{A}(\mathbf{X})$ of open and closed subsets respectively of a given represented space $\mathbf{X}$. For this, we identity an open subset $U$ of $\mathbf{X}$ with its (continuous) characteristic function $\chi_U : \mathbf{X} \to \mathbb{S}$, and a closed subset with the characteristic function of the complement. As countable join (or) and binary meet (and) on $\mathbb{S}$ are computable, we can conclude that open sets are uniformly closed under countable unions, binary intersections and preimages under continuous functions by merely using elementary arguments about function spaces. The space $\mathcal{A}(\mathbf{X})$ corresponds to the upper Fell topology \cite{fell} on the hyperspace of closed sets.

Note that neither negation $\mathalpha{\neg} : \mathbb{S} \to \mathbb{S}$ (i.e.~mapping $\top$ to $\bot$ and $\bot$ to $\top$) nor countable meet (and) $\bigwedge : \mathcal{C}(\mathbb{N},\mathbb{S}) \to \mathbb{S}$ (i.e.~mapping the constant sequence $(\top)_{n \in \mathbb{N}}$ to $\top$ and every other sequence to $\bot$) are continuous or computable operations. They will play the role of fundamental counterexamples in the following. Both operations are equivalent to the \emph{limited principle of omniscience} ($\lpo$) in the sense of Weihrauch reducibility \cite{weihrauchc}.

We need two further hyperspaces, which both will be introduced as subspaces of $\mathcal{O}(\mathcal{O}(\mathbf{X}))$. The space $\mathcal{K}(\mathbf{X})$ of saturated compact sets identifies $A \subseteq \mathbf{X}$ with $\{U \in \mathcal{O}(\mathbf{X}) \mid A \subseteq U\} \in \mathcal{O}(\mathcal{O}(\mathbf{X}))$. Recall that a set is saturated, iff it is equal to the intersection of all open sets containing it (this makes the identification work). The saturation of $A$ is denoted by $\sat{A} := \bigcap \{U \in \mathcal{O}(\mathbf{X}) \mid A \subseteq A\}$. Compactness of $A$ corresponds to $\{U \in \mathcal{O}(\mathbf{X}) \mid A \subseteq U\}$ being open itself. The dual notion to compactness is \emph{overtness}\footnote{This notion is much less known than compactness, as it is classically trivial. It is crucial in a uniform perspective, though. The term \emph{overt} was coined by \name{Taylor} \cite{taylor}, based on the observation that these sets share several closure properties with the open sets.}. We obtain the space $\mathcal{V}(\mathbf{X})$ of overt set by identifying a closed set $A$ with $\{U \in \mathcal{O}(\mathbf{X}) \mid A \cap U \neq \emptyset\} \in \mathcal{O}(\mathcal{O}(\mathbf{X}))$. The space $\mathcal{V}(\mathbf{X})$ corresponds to the lower Fell (equivalently, the lower Vietoris) topology.

Aligned with the definition of the compact and overt subsets of a space, we can also define when a space itself is compact respectively overt:

\begin{definition}
A represented space $\mathbf{X}$ is (computably) compact, iff $\textrm{isFull} : \mathcal{O}(\mathbf{X}) \to \mathbb{S}$ mapping $X$ to $\top$ and any other open set to $\bot$ is continuous (computable). Dually, it is (computably) overt, iff $\textrm{isNonEmpty} : \mathcal{O}(\mathbf{X}) \to \mathbb{S}$ mapping $\emptyset$ to $\bot$ and any non-empty open set to $\top$ is continuous (computable).
\end{definition}

The relevance of $\mathcal{K}(\mathbf{X})$ and $\mathcal{V}(\mathbf{X})$ is found in particular in the following characterizations, which show that compactness just makes universal quantification preserve open predicates, and dually, overtness makes existential quantification preserve open predicates. We shall see later that being Noetherian has the same role for $\Delta^0_2$-predicates.

\begin{proposition}[{\cite[Proposition 40]{pauly-synthetic}}]
\label{prop:exists}
The map $\exists : \mathcal{O}(\mathbf{X} \times \mathbf{Y}) \times \mathcal{V}(\mathbf{X}) \to \mathcal{O}(\mathbf{Y})$ defined by $\exists(R, A) = \{y \in Y \mid \exists x \in A \ (x, y) \in R\}$ is computable. Moreover, whenever $\exists : \mathcal{O}(\mathbf{X} \times \mathbf{Y}) \times \mathcal{S}(\mathbf{X}) \to \mathcal{O}(\mathbf{Y})$ is computable for some hyperspace $\mathcal{S}(\mathbf{X})$ and some space $\mathbf{Y}$ containing a computable element $y_0$, then $\overline{\phantom{A}} : \mathcal{S}(\mathbf{X}) \to \mathcal{V}(\mathbf{X})$ is computable.
\end{proposition}

\begin{proposition}[{\cite[Proposition 42]{pauly-synthetic}}]
\label{prop:forall}
The map $\forall : \mathcal{O}(\mathbf{X} \times \mathbf{Y}) \times \mathcal{K}(\mathbf{X}) \to \mathcal{O}(\mathbf{Y})$ defined by $\forall(R, A) = \{y \in Y \mid \forall x \in A \ (x, y) \in R\}$ is computable. Moreover, whenever $\forall : \mathcal{O}(\mathbf{X} \times \mathbf{Y}) \times \mathcal{S}(\mathbf{X}) \to \mathcal{O}(\mathbf{Y})$ is computable for some hyperspace $\mathcal{S}(\mathbf{X})$ and some space $\mathbf{Y}$ containing a computable element $y_0$, then $\sat{\id} : \mathcal{S}(\mathbf{X}) \to \mathcal{K}(\mathbf{X})$ is computable.
\end{proposition}

\subsubsection*{Connecting computable analysis and topology}
Calling the elements of $\mathcal{O}(\mathbf{X})$ the \emph{open sets} is justified by noting that they indeed form a topology, namely the final topology $X$ inherits from the subspace topology of $\dom(\delta_\mathbf{X})$ along $\delta_\mathbf{X}$. The notion of a continuous map between the represented spaces $\mathbf{X}$, $\mathbf{Y}$ however differs from that of a continuous map between the induced topological spaces. For a large class of spaces, the notions do coincide after all, as observed originally by \name{Schr\"oder} \cite{schroder}.

\begin{definition}
Call $\mathbf{X}$ admissible, if the map $x \mapsto \{U \in \mathcal{O}(\mathbf{X}) \mid x \in U\} : \mathbf{X} \to \mathcal{O}(\mathcal{O}(\mathbf{X}))$ admits a continuous partial inverse.
\end{definition}

\begin{theorem}[{\cite[Theorem 36]{pauly-synthetic}}]
A represented space $\mathbf{X}$ is admissible iff any map $f : \mathbf{Y} \to \mathbf{X}$ is continuous as a map between represented spaces iff it is continuous as a map between the induced topological spaces.
\end{theorem}

The admissible represented spaces are themselves cartesian closed (in fact, it suffices for $\mathbf{Y}$ to be admissible in order to make $\mathcal{C}(\mathbf{X},\mathbf{Y})$ admissible). They can be seen as a joint subcategory of the sequential topological spaces and the represented spaces, and thus form the natural setting for computable topology. They have been characterized by \name{Schr\"oder} as the $\textrm{QCB}_0$-spaces \cite{schroder}, the $\mathrm{T}_0$ quotients of countably based spaces.

\name{Weihrauch} \cite{weihrauchd,weihrauchh} introduced the standard representation of a countably based space: Given some enumeration $(U_n)_{n \in \mathbb{N}}$ of a basis of a topological space $\mathbf{X}$, one can introduce the representation $\delta_\mathbf{B}$ where $\delta_\mathbf{B}(p) = x$ iff $\{n \in \mathbb{N} \mid \exists i \ p(i) = n+1\} = \{n \in \mathbb{N} \mid x \in U_n\}$. This yields an admissible representation, which in turn induces the original topology on $\mathbf{X}$.

Amongst the countably based spaces, the quasi-Polish spaces are distinguished by a completeness properties. We will make use of the following characterization:

\begin{theorem}[{\name{dB} \cite{debrecht6}}]
A topological space $\mathbf{X}$ is quasi-Polish, iff its topology is induced by an open admissible total representation $\delta_\mathbf{X} : \Baire \to \mathbf{X}$.
\end{theorem}

\subsubsection*{Synthetic descriptive set theory}
The central addition of synthetic descriptive set theory (as proposed by the authors in \cite{pauly-descriptive,pauly-descriptive-lics}) is the notion of a computable endofunctor:

\begin{definition}
An endofunctor $d$ on the category of represented spaces is called \emph{computable}, if for any represented spaces $\mathbf{X}$, $\mathbf{Y}$ the induced morphism $d : \mathcal{C}(\mathbf{X},\mathbf{Y}) \to \mathcal{C}(d\mathbf{X},d\mathbf{Y})$ is computable.
\end{definition}

To keep things simple, we will restrict our attention here to endofunctors that do not change the underlying set of a represented spaces, but may only modify the representation. Such endofunctors can in particular be derived from certain maps on Baire space, called \emph{jump operators} by \name{dB.}~in \cite{debrecht5}. Here, we instead adopt the terminology \emph{transparent map} introduced in \cite{gherardi4}. Further properties of transparent maps were studied in \cite{pauly-nobrega-arxiv}.

\begin{definition}
	Call $T : \subseteq \Baire \to \Baire$ \emph{transparent} iff for any computable (continuous) $g : \subseteq \Baire \to \Baire$ there is a computable (continuous) $f : \subseteq \Baire \to \Baire$ with $T \circ f = g \circ T$.
\end{definition}

If the relationship between $g$ and $f$ establishing $T$ to be transparent is uniform, then $T$ will induce a computable endofunctor $t$ by setting $t\mathbf{X}$ to be $(X,\delta_\mathbf{X} \circ T)$, and extending to functions in the obvious way.

By applying a suitable endofunctor to Sierpi\'nski space, we can define further classes of subsets; in particular those commonly studied in descriptive set theory. This idea and its relationship to universal sets is further explored in \cite{pauly-gregoriades}. Basically, we introduce the space $\mathcal{O}^d(\mathbf{X})$ of $d$-open subsets of $\mathbf{X}$ by identifying a subset $U$ with its continuous characteristic function $\chi_U : \mathbf{X} \to d\mathbb{S}$. If $d$ preserves countable products, it automatically follows that the $d$-open subsets are effectively closed under countable unions, binary intersections and preimages under continuous maps. The complements of the $d$-opens are the $d$-closed sets, denoted by $\mathcal{A}^d(\mathbb{S})$.

We will use the endofunctors to generate lifted versions of compactness and overtness:

\begin{definition}
A represented space $\mathbf{X}$ is (computably) $d$-compact, iff $\textrm{isFull} : \mathcal{O}^d(\mathbf{X}) \to d\mathbb{S}$ mapping $X$ to $\top$ and any other open set to $\bot$ is continuous (computable). Dually, it is (computably) $d$-overt, iff $\textrm{isNonEmpty} : \mathcal{O}^d(\mathbf{X}) \to d\mathbb{S}$ mapping $\emptyset$ to $\bot$ and any non-empty open set to $\top$ is continuous (computable).
\end{definition}

A fundamental example of a computable endofunctor linked to notions from descriptive set theory is the limit or jump endofunctor;

\begin{definition}
\label{def:jump}
Let $\lim : \subseteq \Baire \to \Baire$ be defined via $\lim(p)(n) = \lim_{i \to \infty} p(\langle n, i\rangle)$, where $\langle \ , \ \rangle : \mathbb{N} \times \mathbb{N} \to \mathbb{N}$ is a standard pairing function. Define the computable endofunctor $'$ by $(X,\delta_\mathbf{X})' = (X,\delta_\mathbf{X} \circ \lim)$ and the straight-forward lift to functions.
\end{definition}

The map $\lim$ and its relation to the Borel hierarchy and Weihrauch reducibility was studied by \name{Brattka} in \cite{brattka}. The jump of a represented spaces was studied in \cite{ziegler3,gherardi4}. The $'$-open sets are just the $\Sigma^0_2$-sets, and the further levels of the Borel hierarchy can be obtained by iterating the endofunctor.

\subsubsection*{Computability with finitely many mindchanges}
The most important endofunctor for our investigation of Noetherian Quasi-Polish spaces is the finite mindchange endofunctor $\nabla$:
\begin{definition}[\cite{pauly-descriptive}]
Define $\Delta : \subseteq \Baire \to \Baire$ via $\Delta(p)(n) = p(n + 1 + \max \{i \mid p(i) = 0\}) - 1$. Let the finite mindchange endofunctor be defined via $(X, \delta_X)^\nabla = (X, \delta_X \circ \Delta)$ and $(f : \mathbf{X} \to \mathbf{Y})^\nabla = f : \mathbf{X}^\nabla \to \mathbf{Y}^\nabla$.
\end{definition}

We find that $\nabla$ is a monad, and moreover, that $f : \mathbf{X} \to \mathbf{Y}^\nabla$ is computable (continuous) iff $f : \mathbf{X}^\nabla \to \mathbf{Y}^\nabla$ is. The computable maps from $\mathbf{X}$ to $\mathbf{Y}^\nabla$ can equivalently be understood as those maps from $\mathbf{X}$ to $\mathbf{Y}$ that are computable with finitely many mindchanges.

A machine model for computation with finitely many mindchanges is obtained by adding the option of resetting the output tape to the initial state. To ensure that the output is well-defined, such a reset may only be used finitely many times. In the context of computable analysis, this model was studied by a number of authors \cite{ziegler3,debrecht,paulyoracletypetwo,paulybrattka,paulybrattka2,paulyneumann}. For our purposes, an equivalent model based on non-deterministic computation turns out to be more useful. We say that a function from $\mathbf{X}$ to $\mathbf{Y}$ is non-deterministically computable with advice space $\mathbb{N}$, if on input $p$ (a name for some $x \in \mathbf{X}$) the machine can guess some $n \in \mathbb{N}$ and then either continue for $\omega$ many steps and output a valid name for $f(x)$, or at some finite time reject the guess. We demand that for any $p$ there is some $n \in \mathbb{N}$ that is not rejected. The equivalence of the two models is shown in \cite{paulybrattka}.

The interpretation of $\nabla$ in descriptive set theory is related to the $\Delta^0_2$-sets. In particular, the $\nabla$-open sets are the $\Delta^0_2$-sets, the continuous functions from $\mathbf{X}$ to $\mathbf{Y}^\nabla$ are the piecewise continuous functions for Polish $\mathbf{X}$, and the lifted version of admissibility under $\nabla$ corresponds to the Jayne-Rogers theorem (cf.~\cite{jaynerogers,ros,kihara4}). This was explored in detail by the authors in \cite{paulydebrecht}.

\section{$\nabla$-computably Noetherian spaces}
\label{sec:computablynoetherian}
In this section, we want to investigate the notion of being Noetherian in the setting of synthetic topology. We will see that the naive approach fails, but then provide a well-behaved definition. That it is adequate will be substantiated by providing a computable counterpart to the relationship between Noetherian spaces and well-quasiorders. First, however, we will explore a prototypical example.

\subsection{A case study on computably Noetherian spaces}
Let $\mathbb{N}_<$ be the natural numbers with the topology $\mathcal{T}_< := \{L_n := \{i \in \mathbb{N} \mid i \geq n\} \mid n \in \mathbb{N}\} \cup \{\emptyset\}$. Then let $\overline{\mathbb{N}}_<$ be the result of adjoining $\infty$, which is contained in all non-empty open sets. In $\overline{\mathbb{N}}_<$ we find a very simple yet non-trivial example of a quasi-Polish Noetherian space.

Similarly, let $\mathbb{N}_>$ be the natural numbers with the topology $\mathcal{T}_> := \{U_n := \{i \in\mathbb{N} \mid i < n\} \mid n \in \mathbb{N}\} \cup \{\mathbb{N}\}$.  By $\overline{\mathbb{N}}_>$ I denote the space resulting from adjoining an element $\infty$, which is only contained in one open set. In terms of representations, we can conceive of an element in $\mathbb{N}_<$ as being given as the limit of an increasing sequence, and of an element in $\mathbb{N}_>$ as the limit of a decreasing sequence.

Looking at the way how we defined $\mathcal{T}_<$, we see that we have a countable basis, and given indices of open sets, can e.g.~decide subset inclusion. The indexing is fully effective, in the sense that this is a computable basis as follows:

\begin{definition}[{\cite[Definition 9]{pauly-gregoriades}}]
An effective countable base for $\mathbf{X}$ is a computable sequence $(U_i)_{i \in \mathbb{N}} \in \mathcal{C}(\mathbb{N}, \mathcal{O}(\mathbf{X}))$ such that the multivalued partial map $\textrm{Base} :\subseteq \mathbf{X} \times \mathcal{O}(\mathbf{X}) \mto \mathbb{N}$ is computable. Here $\dom(\textrm{Base}) = \{(x, U) \mid x \in U\}$ and $n \in \textrm{Base}(x, U)$ iff $x \in U_n \subseteq U$.
\end{definition}

Even though all open sets are basis elements, we should still distinguish computability on the open sets themselves, and computability on the indices. For example, the map $\bigcup : \mathcal{O}(\mathbf{X})^\mathbb{N} \to \mathcal{O}(\mathbf{X})$, i.e.~the countable union of open sets, should always be a computable operation. This, however, cannot be done on the indices. More generally, in the synthetic topology framework the space of open subsets of a given space automatically comes with its own natural topology. This topology is obtained by demanding that given a point and an open set, we can recognize (semidecide) membership. In the case of $\mathbb{N}_<$, we can establish a quite convenient characterization of its open subsets:

\begin{proposition}
\label{prop:opensets}
The map $n \mapsto \{i \in \mathbb{N} \mid i \geq n\} : \overline{\mathbb{N}}_> \to \mathcal{O}(\mathbb{N}_<)$ is a computable isomorphism.
\begin{proof}
\begin{enumerate}
\item The map is computable.

Given $m \in \mathbb{N}_<$ and $n \in \overline{\mathbb{N}}_>$, we can semidecide $m \geq n$ (just wait until the increasing and the decreasing approximations pass each other).

\item The map is surjective.

At the moment some number $m$ is recognized to be an element of some open set $U \in \mathcal{O}(\mathbb{N}_<)$, we have only learned some lower bound on $m$ so far. Thus, any number greater than $m$ is contained in $U$, too. Hence all open subsets of $\mathbb{N}_>$ are final segments.

\item The inverse of the map is computable.

Given $U \in \mathcal{O}(\mathbb{N}_<)$, we can simultaneously begin testing $i \in U?$ for all $i \in \mathbb{N}$. Any positive test provides an upper bound for the $n$ such that $U =  \{i \in \mathbb{N} \mid i \geq n\}$.
\end{enumerate}
\end{proof}
\end{proposition}

The space of (saturated) compact subsets likewise comes with its own topology, in this case obtained by demanding that given a compact $K$ and an open $U$, we can recognize if $K \subseteq U$. Similarly to the preceding proposition, we can also characterize the compact subsets of $\mathbb{N}_>$:

\begin{proposition}
The map $n \mapsto \{i \in \mathbb{N} \mid i \geq n\} : \overline{\mathbb{N}}_< \to \mathcal{K}(\mathbb{N}_<)$ is a computable isomorphism.
\begin{proof}
\begin{enumerate}
\item The map is computable.

We need to show that given $n \in \overline{\mathbb{N}}_<$ and $U \in \mathcal{O}(\mathbb{N}_>)$ we can recognize that $\{i \in \mathbb{N} \mid i \geq n\} \subseteq U$. By Proposition \ref{prop:opensets}, we can assume that $U$ is of the form $U = \{i \in \mathbb{N} \mid i \geq m\}$ with $m \in \overline{\mathbb{N}}_>$. Now for such $n, m$, we can indeed semidecide $m \leq n$ -- again, just wait until the approximating sequences reach the same value.

\item The map is surjective.

While any subset of $\mathbb{N}_<$ is compact, only the saturated compact sets appear in $\mathcal{K}(\mathbb{N}_>)$, and these are the given ones.

\item The inverse map is computable.

Given a compact set $K \in \mathcal{K}(\mathbb{N}_<)$, we simultaneously test if it is covered by open sets of the form $\{i \mid i \geq m\}$. Any such $m$ we find provides a lower bound for the $n$ for which $K = \{i \mid i \geq n\}$ holds.
\end{enumerate}
\end{proof}
\end{proposition}

So we see that while the {\bf spaces} $\mathcal{O}(\mathbb{N}_<)$ and $\mathcal{K}(\mathbb{N}_<)$ contain the same points, their topologies differ -- and are, in fact, incomparable. There are two potential ways to capture the idea that \emph{opens are compact} in a synthetic way:

We could work with open and compact sets when in a Noetherian space, i.e.~with the space $\mathcal{O}(\mathbb{N}_<) \wedge \mathcal{K}(\mathbb{N}_<)$ carrying the join of the topologies. As $\mathbb{N}_< \wedge \mathbb{N}_> \cong \mathbb{N}$, in this special cases we would end up in the same situation as using computability on base indices straightaway. In general though it is not even obvious if $\cap : \left (\mathcal{O}(\mathbf{X}) \wedge \mathcal{K}(\mathbf{X})\right ) \times \left (\mathcal{O}(\mathbf{X}) \wedge \mathcal{K}(\mathbf{X})\right ) \to \left (\mathcal{O}(\mathbf{X}) \wedge \mathcal{K}(\mathbf{X})\right )$ should be computable.

The second approach relies on the observation that $\mathbb{N}_<$ and $\mathbb{N}_>$ do not differ by \emph{that much}. We can consider \emph{computability with finitely many mindchanges} --  and the distinction between $\mathbb{N}_<$, $\mathbb{N}_>$ and $\mathbb{N}$ disappears, as we find $\mathbb{N}_<^\nabla \cong \mathbb{N}_>^\nabla \cong\mathbb{N}^\nabla$. As the next subsection shows, computability with finitely many mindchanges seems adequate to give \emph{opens are compact} a computable interpretation.

\subsection{The abstract approach}
The straightforward approach to formulate a synthetic topology version of Noetherian would be the following:

\begin{definition}[Hypothetical]
\label{def:hypo}
Call a space $\mathbf{X}$ \emph{computably Noetherian}, iff $\id_{\mathcal{O},\mathcal{K}} : \mathcal{O}(\mathbf{X}) \to \mathcal{K}(\mathbf{X})$ is well-defined and computable.
\end{definition}

This fails entirely, though:

\begin{observation}
Let $\mathbf{X}$ be non-empty. Then $\mathbf{X}$ is not computably Noetherian according to Definition \ref{def:hypo}.
\begin{proof}
Note that $\mathalpha{\subseteq} : \mathcal{K}(\mathbf{X}) \times \mathcal{O}(\mathbf{X}) \to \mathbb{S}$ is by definition of $\mathcal{K}$ a computable map, i.e.~inclusion of a compact in an open set is semidecidable. Furthermore, $\iota : \mathbb{S} \to \mathbf{X}$ defined via $\iota(\top) = X$ and $\iota(\bot) = \emptyset$ is a always a computable injection for non-empty $\mathbf{X}$. Now if $\mathbf{X}$ were computably Noetherian, then the map $t \mapsto \mathalpha{\subseteq}(\id_{\mathcal{O},\mathcal{K}}(\iota(t)), \emptyset)$ would be computable and identical to $\neg : \mathbb{S} \to \mathbb{S}$, but the latter is non-computable.
\end{proof}
\end{observation}

We can avoid this problem by relaxing the computability-requirement to computability with finitely many mindchanges. Now we can try again:

\begin{definition}
\label{def:real}
Call a space $\mathbf{X}$ $\nabla$-\emph{computably Noetherian}, iff $\id_{\mathcal{O},\mathcal{K}} : \mathcal{O}(\mathbf{X}) \to \left (\mathcal{K}(\mathbf{X}) \right )^\nabla$ is well-defined and computable.
\end{definition}

Say that an effective countable base is \emph{nice}, if $\{\langle u, v\rangle \mid \left (U_{u(1)} \cup \ldots \cup U_{u(|u|)} \right ) \subseteq \left (U_{v(1)} \cup \ldots \cup U_{v(|u|)} \right )\} \subseteq \mathbb{N}^* \times \mathbb{N}^*$ is decidable. Clearly any effective countable base is nice relative to some oracle, hence this requirement is unproblematic from the perspective of continuity.

We can now state and prove the following theorem, which can be seen as a uniform counterpart to Theorem \ref{theo:glnoethcharac}:

\begin{theorem}
Let $\mathbf{X}$ be quasi-Polish, and in particular have a nice effective countable base. Then the following are equivalent:
\begin{enumerate}
\item $\mathbf{X}$ is $\nabla$-computably Noetherian
\item $\id_{\mathcal{O},\mathcal{K}} : \mathcal{O}(\mathbf{X}) \to \left (\mathcal{K}(\mathbf{X}) \right )^\nabla$ is well-defined and computable.
\item $\mathalpha{\subseteq} : \mathcal{O}(\mathbf{X}) \times \mathcal{O}(\mathbf{X}) \to \mathbb{S}^\nabla$ is computable.
\item $\operatorname{Stabilize} : \mathcal{C}(\mathbb{N},\mathcal{O}(\mathbf{X})) \mto \mathbb{N}^\nabla$ is well-defined and computable, where $N \in \operatorname{Stabilize}((V_i)_{i \in \mathbb{N}})$ iff $\left (\bigcup_{i = 0}^N V_i \right ) = \left (\bigcup_{i \in \mathbb{N}} V_i \right )$.
\item $\operatorname{Stabilize} : \mathcal{C}(\mathbb{N},\mathcal{A}(\mathbf{X})) \mto \mathbb{N}^\nabla$ is well-defined and computable, where $N \in \operatorname{Stabilize}((A_i)_{i \in \mathbb{N}})$ iff $\left (\bigcap_{i = 0}^N A_i \right ) = \left (\bigcup_{i \in \mathbb{N}} A_i \right )$.
\item The computable map $u \mapsto \left (U_{u(1)} \cup \ldots \cup U_{u(|u|)} \right ) : \mathbb{N}^* \to \mathcal{O}(\mathbf{X})$ is a surjection and has a $\nabla$-computable right-inverse.
\end{enumerate}
Note that the forward implications hold for arbitrary represented spaces, as long as they make sense.
\begin{proof}
\begin{description}
\item[$1. \Leftrightarrow 2.$] This is the definition.
\item[$2. \Rightarrow 3.$] By taking into account the definition of $\mathcal{K}$, we have $\id_{\mathcal{O},\mathcal{K}} : \mathcal{O}(\mathbf{X}) \to \left (\mathcal{C}(\mathcal{O}(\mathbf{X}), \mathbb{S} ) \right )^\nabla$. Moreover, $\id : \mathcal{C}(\mathbf{Y}, \mathbf{Z})^\nabla \to \mathcal{C}(\mathbf{Y}, \mathbf{Z}^\nabla)$ is always computable, so currying yields the claim.
\item[$3. \Rightarrow 4.$] First, we prove that $\operatorname{Stabilize}$ is well-defined. Assume that it is not, then there is a family $(V_i)_{i \in \mathbb{N}}$ of open sets such that $V := \bigcup_{i \in \mathbb{N}} V_i \neq \bigcup_{i = 0}^N V_i$ for all $N \in \mathbb{N}$. Consider the computable map $q \mapsto \mathalpha{\subseteq}\left (V, \bigcup_{i \in \mathbb{N}} V_{q(i)}  \right ) : \Baire \to \mathbb{S}^\nabla$. If the range of $q$ is finite, then the output must be $\bot$, if the range of $q$ is $\mathbb{N}$, then the output must be $\top$. However, these two cases cannot be distinguished in a $\Delta^0_2$-way, thus the $(V_i)_{i \in \mathbb{N}}$ cannot exist, and $\operatorname{Stabilize}$ is well-defined..

    To see that we can compute the (multivalued) inverse, we employ the equivalence to $\nabla$-computability and non-deterministic computation with advice space $\mathbb{N}$ from \cite{paulybrattka}. Given $(V_i)_{i \in \mathbb{N}}$, we guess $N \in \mathbb{N}$ together with an upper bound $b$ on the number of mindchanges happening in verifying that $\mathalpha{\subseteq}(\bigcup_{i = 0}^N V_i, \bigcup_{i \in \mathbb{N}} V_i ) = \top$. Any correct guess contains a valid solution, and any wrong guess can be rejected.
\item[$4. \Leftrightarrow 5.$] By de Morgan's law.
\item[$4. \Rightarrow 6.$] In a quasi-Polish space $\mathbf{X}$ with effectively countable basis $(U_i)_{i \in \mathbb{N}}$, any $U \in \mathcal{O}(\mathbf{X})$ can be effectively represented by $p \in \Baire$ with $U = \bigcup_{i \in \mathbb{N}} U_{p(i)}$. Applying stabilize to the family $(U_{p(i)})_{i \in \mathbb{N}}$ shows subjectivity and computability of the multivalued inverse.
 \item[$6. \Rightarrow 2.$] We will argue that $u \mapsto \left (U_{u(1)} \cup \ldots \cup U_{u(|u|)} \right ) : \mathbb{N}^* \to \mathcal{K}(\mathbf{X})$ is computable, provided that $(U_n)_{n \in \mathbb{N}}$ is a nice basis. For this, note that given $u \in \mathbb{N}^*$ and $p \in \Baire$, we can semidecide whether $\left (U_{u(1)} \cup \ldots \cup U_{u(|u|)} \right ) \subseteq \bigcup_{n \in \mathbb{N}} U_{p(n)}$.
\end{description}
\end{proof}
\end{theorem}

All finite spaces containing only computable points are $\nabla$-computably Noetherian; any quasi-Polish Noetherian space is $\nabla$-computably Noetherian relative to some oracle (which is not vacuous). $\nabla$-computably Noetherian spaces are closed under finite products and finite coproducts, and computable images of $\nabla$-computably Noetherian spaces are $\nabla$-computably Noetherian.

\subsection{Well-quasiorders and $\nabla$-computably Noetherian spaces}
A quasiorder $(X,\preceq)$ can be seen as a topological space via the Alexandrov topology, which consists of the upper sets regarding $\preceq$. The quasiorder is recovered from the topology as the specialization order (i.e. $x \preceq y$ iff $x \in \overline{\{y\}}$). As mentioned in the introduction, the Alexandrov topology of a quasiorder is Noetherian iff the quasiorder is a well-quasiorder. Here, we shall investigate the computability aspects of this connection in the case of countable quasiorders, more precisely, quasiorders over $\mathbb{N}$. We first consider arbitrary quasiorders over $\mathbb{N}$, before coming to the special case of well-quasiorders.

\subsubsection*{Arbitrary quasiorders over $\mathbb{N}$ and their Alexandrov topologies}

\begin{definition}
Given some quasiorder $(\mathbb{N},\preceq)$ we define the represented space $\mathrm{Av}(\preceq)$ to have the underlying set $\mathbb{N}$ and the representation $\psi_\preceq : \subseteq \Baire \to \mathbb{N}$ defined via $\psi_\preceq(p) = \mathbf{n}$ iff: \[\{ k \in \mathbb{N} \mid k \preceq \mathbf{n}\} = \{p(i) \mid i \in \mathbb{N}\}\]
\end{definition}

The represented space $\mathrm{Av}(\preceq)$ corresponds to the Alexandrov-topology induced by $\preceq$. This is seen by the following proposition, which also establishes some basic observations on how computability works in this setting.

\begin{proposition}
\label{prop:quasiorder-opens}
Let $\preceq$ be computable. Then
\begin{enumerate}
\item $\uparrow_\preceq : \mathcal{O}(\mathbb{N}) \to \mathcal{O}(\mathbb{N})$ is computable.
\item $\uparrow_\preceq : \mathcal{O}(\mathbb{N}) \to \mathcal{O}(\mathrm{Av}(\preceq))$ is a computable surjection.
\item $\id : \mathbb{N} \to \mathrm{Av}(\preceq)$ is computable.
\item $\id : \mathcal{O}(\mathrm{Av}(\preceq)) \to \mathcal{O}(\mathbb{N})$ is a computable embedding.
\end{enumerate}
\begin{proof}
\begin{enumerate}
\item Straight-forward.
\item To show that the map is computable, by (1) it suffices to show that given some $\preceq$-upwards closed set $U \in \mathcal{O}(\mathbb{N})$ and $\mathbf{n} \in \mathrm{Av}(\preceq)$, we can semidecide if $\mathbf{n} \in U$. But given the definition of $\mathrm{Av}(\preceq)$, we find that $\psi_\preceq(p) \in U$ iff $\exists i \ p(i) \in U$, hence the semidecidability follows.

    It remains to argue that map is surjective, i.e.~that any $U \in \mathcal{O}(\mathrm{Av}(\preceq))$ is $\preceq$-upwards closed. Assume for the sake of a contradiction that $U \in \mathcal{O}(\mathrm{Av}(\preceq))$ is not upwards-closed, i.e.~that there are $\mathbf{n} \in U$, $\mathbf{m} \notin U$ with $\mathbf{n} \preceq \mathbf{m}$. Pick some $\psi_\preceq$-name $p$ of $\mathbf{n}$ and a realizer $u$ of $\chi_U : \mathrm{Av}(\preceq) \to \mathbb{S}$. Now $u$ will accept the input $p$ after having read some finite prefix $w$ of $p$. Let $q$ be a $\psi_\preceq$-name of $\mathbf{m}$. Now $\psi_\preceq(wq) = \mathbf{m}$, and $u$ will accept $wq$, hence $\mathbf{m} \in U$ follows.
\item Straight-forward.
\item That $\id : \mathcal{O}(\mathrm{Av}(\preceq)) \to \mathcal{O}(\mathbb{N})$ is computable follows from (3). Its computable inverse is given by $\uparrow_\preceq : \mathcal{O}(\mathbb{N}) \to \mathcal{O}(\mathrm{Av}(\preceq)$ from (2).
\end{enumerate}
\end{proof}
\end{proposition}

The terminology \emph{Alexandrov topology} goes back to an observation by Pavel Alexandrov \cite{alexandrov} that certain topological spaces correspond to partial orders, namely those characterized by the property that arbitrary intersections of open sets are open again. Further characterizations are provided in \cite{arenas}. Similar to the failure of the naive Definition \ref{def:hypo}, one can readily check that e.g.~$\bigcap : \mathcal{O}(\mathbf{X})^\mathbb{N} \to \mathcal{O}(\mathbf{X})$ is never a continuous well-defined map, as long as $\mathbf{X}$ is non-empty. Thus, this characterization does not extend to a uniform statement in a straight-forward manner.

Next, we shall explore the compact subsets of $\mathrm{Av}(\preceq)$. As usual in the study of represented spaces, we restrict our attention to the \emph{saturated} compact subsets. Recall that $A \subseteq \mathbf{X}$ is called saturated, iff $A = \bigcap {\{U \in \mathcal{O} \mid A \subseteq U\}}$; the saturation of a set $A$ is $\bigcap {\{U \in \mathcal{O} \mid A \subseteq U\}}$. As a set is compact iff its saturation is, this restriction is without loss of generality. In $\mathrm{Av}(\preceq)$, a set is saturated iff it is upwards closed.

\begin{proposition}
\label{prop:quasiordercompacts}
The map $n_0\ldots n_k \mapsto \uparrow \{n_0,\ldots,n_k\} : \mathbb{N}^* \to \mathcal{K}(\mathrm{Av}(\preceq))$ is a computable surjection.
\begin{proof}
To show that the map is computable, we need to argue that given $n_0\ldots n_k \in \mathbb{N}^*$ and $U \in \mathcal{O}(\mathrm{Av}(\preceq))$, we can semidecide if $\uparrow \{n_0,\ldots,n_k\} \subseteq U$. Since $U$ itself is upwards closed, this is equivalent to $\{n_0,\ldots,n_k\} \subseteq U$. It follows from Proposition \ref{prop:quasiorder-opens} (4) that it is.

To see that the map is surjective, consider some $A \in \mathcal{K}(\mathrm{Av}(\preceq))$. As $A$ is upwards closed, we find that in particular also $A \in \mathcal{O}(\mathrm{Av}(\preceq))$ (in a non-uniform way of course). As $\mathalpha{\subseteq} : \mathcal{K}(\mathbf{X}) \times \mathcal{O}(\mathbf{X}) \to \mathbb{S}$ is computable, we can given the compact set $A$ and the open set $A$ semidecide that indeed $A \subseteq A$. At the moment of the decision, only finite information about the sets has been read. In particular, by Proposition \ref{prop:quasiorder-opens} (4) we can assume that all we have learned about the open set $A$ is $\{n_0,\ldots,n_k\} \subseteq A$ for some finite set $\{n_0,\ldots,n_k\}$. As the semidecision procedure would also accept the compact set $A$ and the open set $\uparrow \{n_0,\ldots,n_k\}$, it follows that $A = \uparrow \{n_0,\ldots,n_k\}$.
\end{proof}
\end{proposition}

Except for trivial cases, the map from the preceding proposition cannot be computably invertible: A compact set $A \in \mathcal{K}(\mathrm{Av}(\preceq))$ can always shrink, whereas each $n_0\ldots n_k \in \mathbb{N}^*$ is completely determined at some finite time. However, moving to computability with finitely many mindchanges suffices to bridge the gap:

\begin{proposition}
\label{prop:basecompactcomputable}
If $\preceq$ is computable, then the multivalued map $\operatorname{Base} : \mathcal{K}(\mathrm{Av}(\preceq)) \mto \left (\mathbb{N}^* \right )^\nabla$ where $n_0\ldots n_k \in \operatorname{Base}(A)$ iff $\uparrow \{n_0,\ldots,n_k\} = A$, is computable.
\begin{proof}
We utilize the equivalence between computability with finitely many mindchanges and non-deterministic computation with discrete advice. Given $A \in \mathcal{K}(\mathrm{Av}(\preceq))$, $n_0\ldots n_k \in \mathbb{N}^*$ and a parameter $t \in \mathbb{N}$ we proceed as follows: If $A \subseteq \uparrow \{n_0,\ldots,n_k\}$ is not confirmed within $t$ steps, reject. If we can find some $m_0,\ldots,m_j$ such that $A \subseteq \uparrow \{m_0,\ldots,m_j\}$ but not $\{n_0,\ldots,n_k\} \subseteq \uparrow \{m_0,\ldots,m_j\}$, then reject.

For fixed $A$ and $n_0\ldots n_k \in \mathbb{N}^*$ , there is a parameter $t \in \mathbb{N}$ not leading to a rejection iff $A = \uparrow \{n_0,\ldots,n_k\}$.
\end{proof}
\end{proposition}

Before we move on to well-quasiorders, we shall consider sobriety for Alexandrov topologies, in light of Proposition \ref{prop:soberquasipolish} and the overall usefulness of sobriety for the results in Section \ref{sec:quasipolish}. Recall that a countable quasiorder $(X, \preceq)$ is a dcpo if for any sequence $(a_i)_{i \in \mathbb{N}}$ in $X$ such that $a_i \preceq a_{i+1}$ we find that there is some $b \in X$ such that for all $c \in X$: \[b \preceq c \Leftrightarrow \forall n \in \mathbb{N} \quad a_n \preceq c\]

\begin{observation}
$\mathrm{Av}(\preceq)$ is sober iff $(\mathbb{N},\preceq)$ is a dcpo.
\end{observation}

\subsubsection*{Well-quasiorders and their Alexandrov topologies}

\begin{proposition}
\label{prop:wqo-opens}
\begin{enumerate}
\item If $\preceq$ is computable, then $n_0\ldots n_k \mapsto \uparrow \{n_0,\ldots,n_k\} : \mathbb{N}^* \to \mathcal{O}(\mathrm{Av}(\preceq))$ is computable.
\item $\preceq$ is a well-quasiorder iff $n_0\ldots n_k \mapsto \uparrow \{n_0,\ldots,n_k\} : \mathbb{N}^* \to \mathcal{O}(\mathrm{Av}(\preceq))$ is a surjection.
\item If $\preceq$ is a computable well-quasiorder, then $\operatorname{Base} : \mathcal{O}(\mathrm{Av}(\preceq)) \mto \left (\mathbb{N}^* \right)^\nabla$ where $n_0\ldots n_k \in \operatorname{Base}(U)$ iff $\uparrow \{n_0,\ldots,n_k\} = U$, is well-defined and computable.
\end{enumerate}
\begin{proof}
\begin{enumerate}
\item This is straight-forward, using Proposition \ref{prop:quasiorder-opens} (4).
\item Let us assume that $U \in \mathcal{O}(\mathrm{Av}(\preceq))$ is not of the form $\uparrow \{n_0,\ldots,n_k\}$. Then in particular, $U \neq \emptyset$. Pick some $a_0 \in U$. As $U \neq \uparrow \{a_0\}$, there is some $a_1 \in U \setminus \{a_0\}$. Subsequently, always pick $a_{n+1} \in U \setminus \uparrow \{a_0,\ldots,a_n\}$. Now $(a_n)_{n \in \mathbb{N}}$ satisfies by construction that for $n < m$ never $a_n \preceq a_m$ holds, i.e.~$(a_n)_{n \in \mathbb{N}}$ is a bad sequence, contradicting the hypothesis $\preceq$ were a wqo.

    Conversely, let $(a_n)_{n \in \mathbb{N}}$ be a bad sequence witnessing that $\preceq$ is not a wqo. Assume that $\uparrow \{a_i \mid i \in \mathbb{N}\} = \uparrow \{n_0,\ldots,n_k\}$. As $n_j \in \uparrow \{a_i \mid i \in \mathbb{N}\}$ for $j \leq k$, there is some $i_j$ such that $n_j \succeq a_{i_j}$. Pick $i_\infty > \max_{j \leq k} i_j$. As $a_{i_\infty} \in \uparrow \{n_0,\ldots,n_k\}$ there is some $j_\infty \leq k$ such that $a_{i_\infty} \succeq n_{j_\infty}$. But then $a_{i_{j_\infty}} \preceq a_{i_\infty}$ follows, and since $i_{j_\infty} < i_\infty$ by construction, this contradicts $(a_i)_{i \in \mathbb{N}}$ being a bad sequence.
\item That the map is well-defined follows from (2). The proof that it is computable is similar to the proof of Proposition \ref{prop:basecompactcomputable}. Given some $U \in \mathcal{O}(\mathrm{Av}(\preceq))$, some $n_0\ldots n_k \in \mathbb{N}^*$ and a parameter $t \in \mathbb{N}$, we test whether for all $j \leq k$ it can be verified in at most $t$ steps that $n_j \in U$, otherwise we reject. In addition, we search for some $a \in U$ such that $n_j \npreceq a$ for all $j \leq k$, if we find one, we reject. For fixed $U$, $n_0\ldots n_k$ there is a value of the parameter $t$ not leading to a rejection iff $\uparrow \{n_0,\ldots,n_k\} = U$.
\end{enumerate}
\end{proof}

The following is the computable counterpart to \cite[Proposition 3.1]{goubault2}. It serves in particular as evidence that our definition of $\nabla$-computably Noetherian is not too restrictive:

\begin{theorem}
Let $\preceq$ be a computable well-quasiorder. Then $\mathrm{Av}(\preceq)$ is $\nabla$-computably Noetherian.
\begin{proof}
By combining Proposition \ref{prop:wqo-opens} (3) with Proposition \ref{prop:quasiordercompacts}, we see that for a computable well-quasiorder $\preceq$ the map $\id : \mathcal{O}(\mathrm{Av}(\preceq)) \to \left (\mathcal{K}(\mathrm{Av}(\preceq)) \right )^\nabla$ is computable.
\end{proof}
\end{theorem}
\end{proposition}

In future work, one should investigate the hyperspace constructions explored in \cite{goubault2} for whether or not they preserve $\nabla$-computable Noetherianess. Research in reverse mathematics has revealed that the preservation of being Noetherian is already equivalent to $\textrm{ACA}_0$ \cite{shafer3}, which typically indicates that $'$ or some iteration thereof is needed, not merely $\nabla$. However, the computational hardness is found in the reverse direction: Showing that if the hyperspace is not Noetherian, then the original well-quasiorder is not computable. Thus, these results merely provide an upper bound on this question in our setting.

\section{Noetherian spaces as $\nabla$-compact spaces}
\label{sec:nablacompact}

For some hyperspace $P(\mathbf{X})$ of subsets of a represented space $\mathbf{X}$, and a space $B$ of truth values $\bot$, $\top$, we define the map $\textrm{isFull} : P(X) \to B$ by $\textrm{isFull}(X) = \top$ and $\textrm{isFull}(A) = \bot$ for $A \neq X$. We recall from \cite{pauly-synthetic} that a represented space is (computably) compact iff $\textrm{isFull} : \mathcal{O}(\mathbf{X}) \to \mathbb{S}$ is continuous (computable).

The space $\mathbb{S}^\nabla \cong \mathbf{2}^\nabla$ can be considered as the space of $\Delta^0_2$-truth values. In particular, we can identify $\Delta^0_2$-subsets of $\mathbf{X}$ with their continuous characteristic functions into $\mathbf{2}^\nabla$, just as the open subsets are identifiable with their continuous characteristic functions into $\mathbb{S}$. By replacing both occurrences of $\mathbb{S}$ in the definition of compactness (one is hidden inside $\mathcal{O}$) by $\mathbb{S}^\nabla$, we arrive at:

\begin{definition}
A represented space $\mathbf{X}$ is called $\nabla$-compact, iff $\textrm{isFull} : \Delta^0_2(\mathbf{X}) \to \mathbb{S}^\nabla$ is computable.
\end{definition}

\begin{theorem}
\label{theo:nablacompact}
A Quasi-Polish space is $\nabla$-compact iff it is $\nabla$-computably Noetherian (relative to some oracle).
\end{theorem}

The proof is provided in the following lemmata and propositions.

Recall that construcible subsets of a topological space are finite boolean combinations of open subsets. For a represented space $\mathbf{X}$, there is an obvious represented space $\const(\mathbf{X})$ of constructible subsets of $\mathbf{X}$: A set $A \in \const(\mathbf{X})$ is given by a (Goedel-number of a) boolean expression $\phi$ in $n$ variables, and an $n$-tuple of open sets $U_1,\ldots,U_n$ such that $A = \phi(U_1,\ldots,U_n)$. Straight-forward calculation shows that we can always assume that $\phi(x_1,\ldots,x_{2n}) = (x_1 \setminus x_2) \cup \ldots \cup (x_{2n-1} \setminus x_{2n})$ without limitation of generality.

\begin{lemma}[(\footnote{This is based on an adaption of the proof of the computable Hausdorff-Kuratowski theorem in \cite{pauly-ordinals-arxiv}.})]
\label{lemma:noetherianconstructible}
Let $\mathbf{X}$ be a $\nabla$-computably Noetherian Quasi-Polish space. Then $\id : \boldd^0_2(\mathbf{X}) \to \const(\mathbf{X})^\nabla$ is well-defined and computable.
\begin{proof}
As $\mathbf{X}$ is Quasi-Polish, we can take it to be represented by an effectively open representation $\delta_\mathbf{X} : \Baire \to \mathbf{X}$.

We can consider our input $A \in \boldd^0_2(\mathbf{X})$ to be given by a realizer $f : \Baire \to \{0,1\}$ of a finite mindchange computation. We consider the positions where a mindchange happens, i.e.~those $w \in \mathbb{N}^*$ which if read by $f$ will cause a mindchange to happen before reading any more of the input. W.l.o.g.~we may assume that the realizer makes at most one mindchange at a given position $w \in \mathbb{N}^*$, and the realizer initially outputs $0$ before reading any of the input.

Let $W \subseteq \mathbb{N}^*$ be the set of mindchange positions. To simplify the following, we will view $\varepsilon$ (the empty string in $\mathbb{N}^*$) as being an element of $W$ (this assumption can be justified formally by viewing the initial output of $0$ as being a mindchange from ``undefined'' to $0$). Note that $W$ is decidable by simply observing the computation of $f$. If we denote the prefix relation on $\mathbb{N}^*$ by $\sqsubseteq$, we see that there are no infinite strictly ascending sequences in $W$ with respect to $\sqsubseteq$, since any such sequence would correspond to an input that induces infinitely many mindchanges. It follows that  $(W, \preceq)$ is a computable total well-order with maximal element $\varepsilon$, where $\preceq$ is the (restriction of the) Kleene-Brouwer order and defined as $v \preceq w$ if and only if  (i) $w \sqsubseteq v$, or (ii) $v(n) < w(n)$, where $n$ is the least position where $v$ and $w$ are both defined and disagree.

We first note that $\min :\subseteq \boldd^0_2(W) \to W$ is $\nabla$-computable, where $\min$ is the function mapping each non-empty $S\in \boldd^0_2(W)$ to the $\preceq$-minimal element of $S$. A realizer for $\min$ on input $S$ can test in parallel whether each element of $W$ is in $S$, and output as a guess the $\preceq$-minimal element which it currently believes to be in $S$. Since $\preceq$ is a well-order and it only takes finitely many mindchanges to determine whether or not a given element is in $S$, this computation is guaranteed to converge to the correct answer.

For each $w\in W$, define $U_w := \bigcup_{v \in W, v \preceq w} \delta_\mathbf{X}[v\Baire]$, which is an effectively open subset of $\mathbf{X}$ and a uniform definition because $\preceq$ is decidable. Next, let $\mathbf{1} = \{*\}$ be the totally represented space with a single point, and define $h \colon W \to (\boldd^0_2(W) + \mathbf{1})$ as $h(w) = *$ if $U_w = \mathbf{X}$ and $h(w) = \{v\in W \mid U_v \subsetneq U_w \}$, otherwise. The computability of the mapping $w \mapsto U_w$ and the assumption that $\mathbf{X}$ is $\nabla$-computably Noetherian implies that it is $\nabla$-decidable whether $U_w = \mathbf{X}$, and also that the characteristic function of the set $\{v\in W \mid U_v \subsetneq U_w \}$ is $\nabla$-computable given $w\in W$. It follows that $h$ is well-defined and $\nabla$-computable.

We construct a finite sequence $v_0 \prec \ldots \prec v_k$ in $W$ by defining $v_0= \min(W)$ and $v_{n+1}=\min(h(v_n))$ whenever $h(v_n)\not=*$. This sequence is necessarily finite because the $U_{v_n}$ form a strictly increasing sequence of open sets and $\mathbf{X}$ is Noetherian. Note that the last element $v_k$ in the sequence satisfies $h(v_k)=*$. It follows that the sequence $\langle v_0,\ldots, v_k\rangle \in W^*$ can be $\nabla$-computed from the realizer $f$ because it only involves a finite composition of $\nabla$-computable functions, and it can be $\nabla$-decided when the sequence terminates.

Define $\eta \colon W \to\{0,1\}$ to be the computable function mapping each $w\in W$ to the output of the realizer $f$ after the mindchange upon reading $w$ (thus $\eta(\varepsilon)=0$). For $n\leq k$ define $V_n := U_{v_n} \setminus \bigcup_{m<n} U_{v_m}$. We claim that $A = \bigcup\{V_n \mid 0\leq n \leq k \text{ \& } \eta(v_n)=1\}$, from which it will follow that we can $\nabla$-compute a name for $A \in \const(\mathbf{X})$ from the realizer $f$.

Fix $x\in\mathbf{X}$, and let $w \in W$ be $\preceq$-minimal such that $x\in \delta_\mathbf{X}[w\Baire]$. It follows that $x\in A$ if and only if $\eta(w) = 1$, because $w$ is a prefix of some name $p$ for $x$, and the $\preceq$-minimality of $w$ implies that the realizer $f$ does not make any additional mindchanges on input $p$ after reading $w$. Next, let $n \in \{0,\ldots, k\}$ be the least number satisfying $x\in V_n$. It is clear that $w \preceq v_n$. Conversely, if $n=0$ then $v_n = v_0 \preceq w$ by the $\preceq$-minimality of $v_0$. If $n>0$, then $w \not\preceq v_{n-1}$ hence $x$ is a witness to $U_{v_{n-1}} \subsetneq U_w$, which implies $v_n =h(v_{n-1}) \preceq w$. Thus $w=v_n$, and it follows that $x\in A$ if and only if $x \in \bigcup\{V_n \mid 0\leq n \leq k \text{ \& } \eta(v_n)=1\}$, which completes the proof.
\end{proof}
\end{lemma}

\begin{proposition}
\label{prop:constructibleisfull}
Let $\mathbf{X}$ be $\nabla$-computably Noetherian. Then $\textrm{isFull} : \const(\mathbf{X}) \to \2^\nabla$ is computable.
\begin{proof}
It is well-known that the sets in $\const(\mathbf{X})$ have a normal form $A = (U_0 \setminus V_0) \cup \ldots \cup (U_n \setminus V_n)$, and this is obtainable uniformly. Now $A = X$ iff $\forall I \subseteq \{0,\ldots,n\} \ \left ( \bigcap_{j \notin I} V_j \right ) \subseteq \left ( \bigcup_{i \in I} U_i \right )$. To see this, first note that the special case $I = \{0,\ldots,n\}$ yields $X = \bigcup_{i \in I} U_i$. Now consider for each $x \in X$ the statement for $I = \{i \mid x \notin V_i\}$.

In a $\nabla$-computably Noetherian space, we can compute $\left ( \bigcap_{j \notin I} V_j \right ) $ as a compact set, and decide its inclusion in $\left ( \bigcup_{i \in I} U_i \right )$ with finitely many mindchanges. Doing this for the finitely many choices of $I$ is unproblematic, thus yielding the claim.
\end{proof}
\end{proposition}

\begin{proposition}
\label{prop:notnablacompact}
Let $\mathbf{X}$ admit a partition $(A_n)_{n \in \mathbb{N}}$ into non-empty $\boldd^0_2$-sets. Then $\mathbf{X}$ is not $\nabla$-compact.
\begin{proof}
Given some $(t_i)_{i \in \mathbb{N}} \in (\2^\nabla)^\mathbb{N}$, we can compute the set $A := \{x \in \mathbf{X} \mid \exists n \in \mathbb{N} \ x \in A_n \wedge t_n = 1\} \in \boldd^0_2(\mathbf{X})$. If $\mathbf{X}$ were $\nabla$-compact, then applying $\textrm{isFull} : \boldd^0_2(\mathbf{X}) \to \mathbb{S}^\nabla \cong \2^\nabla$ to $A$ would yield a computable realizer of $\bigwedge : (\2^\nabla)^\mathbb{N} \to \2^\nabla$.
\end{proof}
\end{proposition}

\begin{proof}[Proof of Theorem \ref{theo:nablacompact}]
By combining Lemma \ref{lemma:noetherianconstructible} and Proposition \ref{prop:constructibleisfull}, we see that for a $\nabla$-computably Noetherian quasi-Polish space $\mathbf{X}$ the map $\textrm{isFull} : \Delta^0_2(\mathbf{X}) \to \mathbb{S}^\nabla$ is computable, i.e.~it is $\nabla$-compact. Conversely, if $\mathbf{X}$ is not Noetherian, then by Lemma \ref{lemma:converingcoverse} there is a countably-infinite $\Delta^0_2$-partition of $\mathbf{X}$, so by Proposition \ref{prop:notnablacompact}, it cannot be $\nabla$-compact.
\end{proof}

The significance of $\nabla$-compactness and Theorem \ref{theo:nablacompact} lies in the following proposition that supplies the desired quantifier-elimination result. The proof is a straight-forward adaption of the corresponding result for compact spaces and open predicates from \cite{pauly-synthetic} (recalled here as Propositions \ref{prop:exists},\ref{prop:forall}), which in turn has \cite{escardo2} and \cite{nachbin} as intellectual predecessors. Note that as $\neg : \mathbb{S}^\nabla \to \mathbb{S}^\nabla$ is computable, it follows that $\nabla$-compactness and $\nabla$-overtness coincide:

\begin{proposition}
\label{prop:quantifierelimination}
The following are equivalent for a represented space $\mathbf{X}$:
\begin{enumerate}
\item $\mathbf{X}$ is $\nabla$-compact.
\item For any represented space $\mathbf{Y}$, the map $\mathalpha{\forall} : \Delta^0_2(\mathbf{X} \times \mathbf{Y}) \to \Delta^0_2(\mathbf{Y})$ mapping $R$ to $\{y \in \mathbf{Y} \mid \forall x \in \mathbf{X} \ (x,y) \in R\}$ is computable.
\item For any represented space $\mathbf{Y}$, the map $\mathalpha{\exists} : \Delta^0_2(\mathbf{X} \times \mathbf{Y}) \to \Delta^0_2(\mathbf{Y})$ mapping $R$ to $\{y \in \mathbf{Y} \mid \exists x \in \mathbf{X} \ (x,y) \in R\}$ is computable.
\end{enumerate}
\end{proposition}

\begin{corollary}
\label{corr:quantifier}
A formula built from $\Delta^0_2$-predicates, boolean operations and universal and existential quantification over Noetherian quasi-Polish spaces defines itself a $\Delta^0_2$-predicate.
\end{corollary}

\begin{corollary}
Let $\mathbf{X} = \mathbf{X}_0 \times \ldots \times \mathbf{X}_n$ be a Noetherian Quasi-Polish space. If a subset $U \subseteq \mathbf{X}_0$ is definable using a finite expression involving open predicates in $\mathbf{X}$, boolean operations, and existential and universal quantification, then $U$ is definable using a finite expression involving open predicates in $\mathbf{X}_0$ and boolean operations.
\begin{proof}
Combine Corollary \ref{corr:quantifier} and Lemma \ref{lemma:noetherianconstructible}.
\end{proof}
\end{corollary}
\section{Other compactness and overtness notions}
\label{sec:othernotions}
It is a natural question whether further lifted counterparts of compactness and overtness might coincide with familiar notions from topology. We will in particular explore this for the $'$-endofunctor from Definition \ref{def:jump}. One such result was already obtained before:

\begin{theorem}[{\cite[Theorem 42]{paulydebrecht2-lics}}]
A Polish space is $\sigma$-compact iff it is $'$-overt.
\end{theorem}

For our remaining investigations, we will rely on the following lemma:

\begin{lemma}
\label{lem:genericounterexample}
Let $d$ be a computable endofunctor such that $\bigvee : \mathcal{C}(\mathbb{N},d\mathbb{S}) \to \mathbb{S}$ and $\wedge : d\mathbb{S} \times d\mathbb{S} \to d\mathbb{S}$ are computable, but $\bigwedge : \mathcal{C}(\mathbb{N},d\mathbb{S}) \to \mathbb{S}$ is not continuous. Then if $\mathbf{X}$ admits a partition into countably-infinitely many non-empty $d$-open subsets, $\mathbf{X}$ is not $d$-compact.
\begin{proof}
Let $(U_n)_{n \in \mathbb{N}}$ be a partition into countably-many $d$-open sets. From the computability of $\bigvee : \mathcal{C}(\mathbb{N},d\mathbb{S}) \to \mathbb{S}$ and and $\wedge : d\mathbb{S} \times d\mathbb{S} \to d\mathbb{S}$  we can conclude that $(b_n)_{n \in \mathbb{N}} \mapsto \bigcup_{\{i \in \mathbb{N} \mid b_i = \top\}} U_i : \mathcal{C}(\mathbb{N},d\mathbb{S}) \to \mathcal{O}^d(\mathbf{X})$ is continuous. If $\mathbf{X}$ were $d$-compact, then we could apply $\textrm{isFull} : \mathcal{O}^d(\mathbf{X}) \to d\mathbb{S}$ to the resulting set, and would obtain $\bigwedge_{n \in \mathbb{N}} b_n \in d\mathbb{S}$, in contradiction to our assumption.
\end{proof}
\end{lemma}

\begin{proposition}
A quasi-Polish space is $\nabla$-compact iff it is $'$-compact.
\begin{proof}
Assume that a quasi-Polish space $\mathbf{X}$ is $\nabla$-compact. Let $U \in \mathcal{O}'(\mathbf{X})$ be a $\Sigma^0_2$-set. This can be effectively written as $U = \bigcup_{n \in \mathbb{N}} U_n$ with disjoint $\Delta^0_2$-sets $U_n$. As $\bigvee : \mathcal{C}(\mathbb{N},\mathbb{S}^\nabla) \to \mathbb{S}'$ is computable, we can compute $\bigvee_{N \in \mathbb{N}} \textrm{isFull}(\bigcup_{n \leq N} U_n)$ using $\textrm{isFull} : \mathcal{O}^\nabla(\mathbf{X}) \to \mathbb{S}^\nabla$. If this yields $\top$, then clearly $U = X$. Conversely, if $U = X$, then by Theorem \ref{theo:finitecover}, already $U = U_N = X$ for sufficiently large $N$, hence the procedure yields $\top$. It follows that $\mathbf{X}$ is $'$-compact.

Now assume that a quasi-Polish space $\mathbf{X}$ is not $\nabla$-compact. Then by Theorem \ref{theo:nablacompact} it is not Noetherian, hence by Theorem \ref{theo:finitecover} there is an infinite $\Delta^0_2$-cover $(U_n)_{n \in \mathbb{N}}$ without a finite subcover. We can refine this into a $\Delta^0_2$-partition (which of course is also a $\Sigma^0_2$-partition). Note that Lemma \ref{lem:genericounterexample} applies to $'$, hence $\mathbf{X}$ is not $'$-compact.
\end{proof}
\end{proposition}

\begin{proposition}
A quasi-Polish space is $''$-compact relative to some oracle iff it is finite.
\begin{proof}
A finite quasi-Polish space is $''$-compact relative to an oracle enumerating the points, as $\wedge : \mathbb{S}'' \times \mathbb{S}'' \to \mathbb{S}''$ is computable. Conversely, any singleton $\{x\}$ in a quasi-Polish space is $\Pi^0_2$, hence also $\Sigma^0_3$. If $\mathbf{X}$ is an infinite quasi-Polish space, we can thus find a proper countably-infinite $\Sigma^0_3$-partition. By Lemma \ref{lem:genericounterexample}, it can then not be $''$-compact.
\end{proof}
\end{proposition}

One could also start the search from the other direction, by exploring some variations on compactness from topology. We conclude by listing some potentially promising examples.

\begin{definition}[\cite{menger}]
A topological space $\mathbf{X}$ is called \emph{Menger}, if for any sequence $(\mathcal{U}_{n \in \mathbb{N}})$ of open covers of $\mathbf{X}$ there exists finite subsets $\mathcal{V}_n \subseteq \mathcal{U}_n$ such that $\bigcup_{n \in \mathbb{N}} \mathcal{V}_n$ is an open cover of $\mathbf{X}$.
\end{definition}

It had been asked by \name{Hurewicz} whether Menger spaces might coincide with the $\sigma$-compact ones \cite{hurewicz3}. The two notions were conditionally separated by \name{Miller} and \name{Fremlin} \cite{fremlin}, and then unconditionally by Bartoszynski and Tsaban \cite{tsaban}. A similar property is named after Hurewicz: In a Hurewicz space, the cover $\bigcup_{n \in \mathbb{N}} \mathcal{V}_n$ needs to have the property that any point belongs to all but finitely many sets from the cover. Both the Menger and the Hurewicz property are special cases of selection principles as identified by \name{Scheepers} \cite{scheepers}. These might provide a fruitful hunting ground for further topological properties corresponding to relativized compactness or overtness notions.

\section*{Acknowledgements}
We are grateful to the participants of the Dagstuhl-seminar \emph{Well-quasi orders in Computer Science} for valuable discussions and inspiration (cf.~\cite{dagstuhlwqo}). In particular we would like to thank Jean Goubault-Larrecq and Takayuki Kihara. Boaz Tsaban suggested the investigation of Menger and Hurewicz spaces in this context to us. The second author thanks Paul Shafer for explaining results pertaining to Noetherian spaces in reverse mathematics.

\bibliographystyle{eptcs}
\bibliography{../../spieltheorie}
\end{document}